\documentclass{mathscan}%

\usepackage{lscape}
\usepackage{setspace}
\usepackage{amssymb,amsfonts,amsmath,amsthm}
\usepackage{cite, url} 
\usepackage{tikz}
\usepackage{xy} 
\usepackage{rotating}
\usepackage{bbm}
\xyoption{all} 
\newdir{ >}{{}*!/-12pt/@{>}} 
\newdir{ -}{{}*!/-5pt/@{}} 
\newdir{- }{{}*!/-5pt/@{}} 
\newdir{ ^(}{{}*!/-5pt/@^{(}} 
\newdir{ _(}{{}*!/-5pt/@_{(}}
\newdir{ |}{{}*!/-5pt/@{|}}
\newdir{_(}{{}*!/0pt/@_{(}}
\xyoption{arc}
\xyoption{ps}
\AtBeginDocument{%
  \setlength{\abovedisplayskip}{4pt}%
  \setlength{\belowdisplayskip}{4pt}%
}

\newcommand{\grad}{\mathrm{grad}}
\newcommand{\Span}{\mathrm{Span}}
\newcommand{\Tr}{\mathrm{Tr}}
\newcommand{\type}{\mathrm{type}}

\newcommand{\twocomp}{{}^{{\kern -.7pt}\wedge}_2}
\newcommand{\ellcomp}{{}^{{\kern -.5pt}\wedge}_{\ell}}

\DeclareMathOperator{\im}{im}

\newcommand{\Hcomp}{{}^{{\kern -.4pt}\wedge}_{{\kern -.5pt}H}}

\newcommand{\GW}{\mathrm{GW}}
\newcommand{\MW}{\mathrm{MW}}

\theoremstyle{definition} 
\newtheorem{proposition}[equation]{Proposition}
\newtheorem{definition}[equation]{Definition} 
\newtheorem{theorem}[equation]{Theorem}
\newtheorem{example}[equation]{Example} 
\newtheorem{corollary}[equation]{Corollary} 
 
\newtheorem{lemma}[equation]{Lemma} 
\theoremstyle{remark} 
\newtheorem{remark}[equation]{Remark}

\newcommand{\bbA}{\mathbb{A}} 

\newcommand{\bbC}{\mathbb{C}}

\newcommand{\bbF}{\mathbb{F}} 

\newcommand{\bbH}{\mathbb{H}}

\newcommand{\bbP}{\mathbb{P}}
\newcommand{\bbQ}{\mathbb{Q}}
\newcommand{\bbR}{\mathbb{R}}

\newcommand{\bbZ}{\mathbb{Z}}

\newcommand{\frakm}{\mathfrak{m}}

\begin{document}

\title{Representability of the local motivic Brouwer degree}
\author{Gereon Quick}
\email{gereon.quick@ntnu.no}
\author{Therese Strand}
\email{therese.strand@ntnu.no}
\author{Glen Matthew Wilson}
\email{glen.m.wilson@ntnu.no}
\address{Department of Mathematics, NTNU, Norway.}

\date{\today}

\begin{abstract}
We study which quadratic forms are representable as the local degree of a map $f : \bbA^n \to \bbA^n$ with an isolated zero at $0$, following the work of Kass and Wickelgren who established the connection to the quadratic form of Eisenbud, Khimshiashvili, and Levine. Our main observation is that over some base fields $k$, not all quadratic forms are representable as a local degree. Empirically the local degree of a map $f : \bbA^n \to \bbA^n$ has many hyperbolic summands, and we prove that in fact this is the case for local degrees of low rank. We establish a complete classification of the quadratic forms of rank at most $7$ that are representable as the local degree of a map over all base fields of characteristic different from $2$. The number of hyperbolic summands was also studied by Eisenbud and Levine, where they establish general bounds on the number of hyperbolic forms that must appear in a quadratic form that is representable as a local degree. Our proof method is elementary and constructive in the case of rank 5 local degrees, while the work of Eisenbud and Levine is more general. We provide further families of examples that verify that the bounds of Eisenbud and Levine are tight in several cases. 
\end{abstract}

\maketitle

\section{Introduction}

The degree of a map in differential topology is a homotopy invariant that completely classifies the homotopy classes of maps between spheres of the same dimension. The work of Morel \cite[Corollary 1.24]{Morel12} establishes an analog of this calculation in the setting of motivic homotopy theory, where algebraic varieties like $\bbP^1$ and $\bbA^n\setminus \{0\}$ can be viewed as spheres. Morel replaces the degree of a map between spheres with the $\bbA^1$-degree, $\deg^{\bbA^1}$, a motivic version that takes values in the Grothendieck--Witt ring of the base field $k$, rather than merely the ring of integers. In differential topology, the degree of a map $f : S^n \to S^n$ can be calculated as a sum of local degrees and one must wonder if an analogous result could also be true in motivic homotopy theory as well. 

If one assumes the existence of a global degree map $\deg : [S^n, S^n] \to \bbZ$, it is not hard to give an abstract definition of a local degree in topology. By working with coordinate charts, one can reduce the problem to defining the local degree of a map $f : \bbR^n \to \bbR^n$ with an isolated zero at $x \in \bbR^n$. From such a map, we can find some $\epsilon > 0$ and construct $\frac{f-f(x)}{\lvert f-f(x) \rvert} : S^{n-1}(x;\epsilon) \to S^{n-1}$. We then define the local degree of $f$ at $x$ to be the global degree of $\frac{f-f(x)}{\lvert f - f(x) \rvert}$. 

A modified construction works in motivic homotopy theory, which was studied by Lannes and Morel (see \cite{WilliamsWickelgren}) and Kass and Wickelgren \cite{KassWickelgren}. 
Let $k$ be a field with characteristic different from 2. Write $\MW(k)$ for the semiring of isometry classes of non-degenerate quadratic forms over the field $k$ and $\GW(k)$ for the group completion of $\MW(k)$, i.e., the Grothendieck--Witt group of $k$. We will use the notation $\langle a_1,...,a_n\rangle$ for the element of $\GW(k)$ represented by the quadratic form $q(x) = \sum_{i=1}^n a_ix_i^2$ where $a_i \in k^{\times}$ for $1\leq i \leq n$.  
For any map $f : \bbA^n \to \bbA^n$ with an isolated zero at $0$, Kass and Wickelgren \cite[Definition 11]{KassWickelgren} construct a map $f_0 : \bbP^n/\bbP^{n-1} \to \bbP^n/\bbP^{n-1}$ in the motivic homotopy category. The space $\bbP^n/\bbP^{n-1}$ is a sphere in motivic homotopy theory and the map $f_0$ encodes the local behavior of $f$ at $0$, just like the local degree in algebraic topology. We can then apply Morel's $\bbA^1$-Brouwer degree $\deg^{\bbA^1} : [\bbP^n/\bbP^{n-1}, \bbP^n/\bbP^{n-1}] \to \GW(k)$ to $f_0$ to obtain the class of a quadratic form in $\GW(k)$ that acts as the local degree of $f$ at $0$ in motivic homotopy theory. We write $\deg_0^{\bbA^1}(f) = \deg^{\bbA^1}(f_0)$ for this motivic local degree. 
The main result of Kass and Wickelgren in \cite{KassWickelgren} is that the local degree $\deg^{\bbA^1}_0(f)$ is equal to the class of the Eisenbud--Khimshiashvili--Levine form of $f$ at $0$ (henceforth EKL form of $f$ at $0$) in $\GW(k)$. 
With this result, one can turn algebraic results about EKL forms into statements about Morel's degree homomorphism and motivic homotopy theory.
In this paper, we study the structure of EKL forms of a map $f : \bbA^n \to \bbA^n$ with isolated zero at $0$ and give several applications to motivic homotopy theory and the study of real singularities.

In topology it is simple to come up with explicit maps $f : \bbR^n \to \bbR^n$ when $n \geq 2$ with an isolated zero at 0 that have a prescribed local degree $\deg_0 (f) = n$. In this paper, we show that a similar statement in motivic homotopy theory is false over a general field $k$. 
Precisely, for an arbitrary element $q \in \MW(k)$, the equation $\deg^{\bbA^1}_0(f) = q$ does not always admit a solution $f : \bbA^n \to \bbA^n$. Of course, the solvability of these equations depends on the base field. For instance, if $-1$ is not a square in the field $k$, then we show that the quadratic form $q(x,y) = x^2 + y^2$, is never representable as the local degree of a map $f: \bbA^n \to \bbA^n$ at $0$. In fact, we show something much stronger: every quadratic form with rank at least 2 that is representable as a local degree contains the hyperbolic form $\bbH(x,y) = x^2 - y^2$ as a direct summand. 

This theorem is proven by proving the corresponding result for EKL forms at $0$. That is, every EKL form at $0$ that has rank at least 2 contains $\bbH$ as a direct summand, see Theorem \ref{thm:H-summand}. This result, coupled with some explicit examples, gives a complete description of which quadratic forms of rank at most $4$ are representable as an EKL form at $0$. It takes us considerably more effort to produce a complete classification of EKL forms of rank 5, which is Theorem \ref{thm:rank-5}: Any quadratic form of rank 5 that is representable as a local degree at $0$ is of the form $2\bbH + \langle a \rangle$ for some unit $a$. This result then can be translated into motivic homotopy theory using Morel's degree isomorphism to say something about which maps of spheres are representable as local degree maps. See Theorem \ref{thm:maps-rank-5} for a precise statement. 
The main tool we use in the classification of the rank 5 EKL forms is a dimension reduction argument, stated in Theorem \ref{thm:dimension_reduction}. This generalizes an observation McKean had in the 2-variable case \cite[Lemma 5.7]{Mckean}. The effect of our dimension reduction result allows us to reduce the study of rank 5 EKL forms to maps $f: \bbA^2 \to \bbA^2$. This can then be checked directly by brute force using some properties of Gr{\"o}bner bases. 

After discovering the results above about the representability of quadratic forms as local degrees in motivic homotopy theory, we found that an analogous question was originally studied by Eisenbud and Levine, using a technical result of Teissier in their paper \cite{ELT}. 
In this paper, Eisenbud and Levine tried to understand how the local degree of a smooth map germ $f : \bbR^n \to \bbR^n$ could be determined from the structure of the local ring $Q_0(f)$. 
They show that the topological local degree of $f$ at $0$ is the signature of the EKL form of $f$ at $0$ and then investigate what kind of restrictions there are on the signature of an EKL form when the dimension of $Q_0(f)$ is known. 
Without explicitly stating it, they are simply obtaining bounds on the number of hyperbolic summands that the EKL form must have for a fixed rank---exactly the question we set out to study, albeit in a slightly different context. 
Eisenbud and Levine ultimately establish a more general result on the minimal number of hyperbolic forms that must appear in an EKL form of fixed rank using a fundamental inequality of Teissier \cite[Appendix, p.\ 38]{ELT}. 
They show that an EKL form of rank $N$ arising from a map $f : \bbA^n \to \bbA^n$ must have at least $\frac{N - N^{1-1/n}}{2}$ hyperbolic summands \cite[Theorem 3.9(i)]{ELT}. 
They also include the weaker statement without a dependence on the dimension $n$ that an EKL form of rank $N$ must contain at least $\frac{N}{4}$ hyperbolic summands \cite[Theorem 3.9(ii)]{ELT}.
Their general inequality gives an alternative proof that a rank 5 EKL form must be of the form $2\bbH +\langle a \rangle$. 

We are pleased to see that our observations about the number of hyperbolic summands in EKL forms can indeed be generalized and refined to the inequalities of Eisenbud, Levine, and Teissier \cite{ELT}. We believe our method of proof for establishing the exact number of hyperbolic forms that must appear in an EKL form of rank 5 can be of use to tighten the inequalities of Eisenbud and Levine. We hope that our observation that not all automorphisms of spheres are representable as local degrees in motivic homotopy theory is of interest to other researchers in the field. 

We do believe we contribute some new results to the subject. It appears that our verification that the bounds of Eisenbud and Levine are tight in the specific case of EKL forms arising from maps of the plane $ f : \bbA^2 \to \bbA^2$ with rank $4^n$ and $9^n$ is new. The chain rule is the key idea to produce the examples that show the bounds are tight. An algebraic proof of the chain rule was given by \cite{KST_published}, which we revisit with a few minor corrections. We also obtain a complete classification of the quadratic forms that are realizable as local degrees over finite fields of odd characteristic. The applications listed to local degrees of real maps were certainly known to Eisenbud and Levine \cite{ELT} and Arnold \cite{Arnold-1978}, but we include them to round out the story.

\subsection*{Notation}
We make heavy use of the notation and results from the work of Kass and Wickelgren \cite{KassWickelgren}, and refer the reader there for a more thorough overview. Write $\bbA^n=\bbA^n_k$ for affine $n$-space over the field $k$. If $f \colon \bbA^n \to \bbA^n$ is a morphism of affine spaces over $k$, we represent $f$ with coordinate functions $f = (f_1,...,f_n)$ where each $f_i$ is a polynomial $f_i \in k[x_1,...,x_n]$. We will write $P = k[x_1,...,x_n]$ for the polynomial ring over $k$ in $n$ variables and $\frakm = (x_1,...,x_n)$ for the maximal ideal corresponding to the origin. The local ring of $\bbA^n$ at $0$ is $P_0 = k[x_1,...,x_n]_\frakm$, and we write $Q_0(f) = P_0/(f_1,...,f_n)$ for the local ring of $f$ at $0$. The map $f$ is said to have an isolated zero at $0$ if $Q_0(f)$ is a finite dimensional vector space over $k$. We write $E=E(f)=E_0(f) \in Q_0(f)$ for the distinguished socle element of $f$ at the origin, which may be described as follows. First find expressions $f_i(x) = f_i(0) +  \sum_{j=1}^n a_{ij}x_j$ with polynomials $a_{ij}\in P$, then define $E$ to be the image of $\det(a_{ij})$ in $Q_0(f)$.

\section{Computations of local degrees}
\label{sec:computations}

Let $k$ denote a field with characteristic different from 2. 
Kass and Wickelgren remark in \cite[Remark 2 and Lemma 4]{KassWickelgren} that the socle of $Q_0(f)$ is equal to the annihilator of the maximal ideal. Moreover, $E$ generates the socle of $Q_0(f)$ if $f$ has an isolated zero at the origin. 
This enables Kass and Wickelgren to make the following definition \cite[Lemma 6 and Definition 7]{KassWickelgren}.

\begin{definition}
\label{def:EKL}
Let $f : \bbA^n \to \bbA^n$ be a map with an isolated zero at $0$ and consider a linear function $\phi : Q_0(f) \to k$ that satisfies $\phi(E) = 1$. Define the symmetric bilinear $\beta_\phi$ on $Q_0(f)$ by the formula $\beta_\phi(x,y) = \phi(x\cdot y)$. The EKL form of $f$ at $0$, denoted by $w_0(f)$, is the class of the symmetric bilinear form $\beta_{\phi}$ on $Q_0(f)$ in $\GW(k)$.
\end{definition}

The EKL form described above is in fact well defined. Any two linear functions $\phi$ and $\phi'$ which satisfy $\phi(E) = \phi'(E) = 1$ yield isometric bilinear forms $\beta_\phi \cong \beta_{\phi'}$. Furthermore, when $\phi(E)=1$ it follows that the form $\beta_{\phi}$ is non-degenerate. We say the rank of $w_0(f)$ is the rank of the symmetric bilinear form $\beta_{\phi}$, which is just the dimension of $Q_0(f)$ as a $k$-vector space. As we assume the characteristic of $k$ is not 2, symmetric bilinear spaces and quadratic spaces are equivalent; we freely consider $\beta_\phi$ as both a symmetric bilinear form and a quadratic form without any distinction. As a quadratic form, $\beta_\phi$ is defined as $\beta_\phi(x) = \phi(x\cdot x)$.

The observation that $E$ generates the annihilator of $Q_0(f)$ leads to our first result. 

\begin{theorem}
\label{thm:H-summand}
Let $f \colon \bbA^n \to \bbA^n$ be a map with an isolated zero at $0$. If the dimension of $Q_0(f)$ as a $k$-vector space is at least 2, then $\bbH$ is a direct summand of the EKL class $w_0(f)$.
\end{theorem}

\begin{proof}
The distinguished socle element $E \in Q_0(f)$ generates the annihilator of the ideal $\frakm = (x_1,...,x_n)$ in $Q_0(f)$. If we assume that $E \notin \frakm$, then $E$ is a unit in $Q_0(f)$. So as $E$ annihilates each $x_i$, the equations $E\cdot x_i = 0$ in $Q_0(f)$ imply $x_i = 0$ in $Q_0(f)$. Hence $Q_0(f)$ is spanned by $1$, and $\dim_k Q_0(f) \leq 1$. We therefore conclude $E \in \frakm$. Hence in choosing a $k$-linear function $\phi : Q_0(f) \to k$ we are free to take $\phi(1)=0$. Recall that $\beta_\phi$ is a non-degenerate quadratic form and observe that $\beta_\phi(1) = \phi(1\cdot 1) = 0$, that is, $1$ is a non-trivial isotropic vector. Witt proves in \cite[Satz 5]{Witt} that under these circumstances, the quadratic form $\beta_\phi$ admits $\bbH$ as a direct summand. This completes the proof. 
\end{proof}

In Lemmas \ref{lem:2dim-1} through \ref{lem:n-m-powers} we collect some basic calculations of EKL forms. These have appeared in the work of others, such as Pauli, Kass--Wickelgren, McKean, etc. We include the statements and proofs for completeness as they provide, in particular, concrete examples of EKL forms in low degrees. 
Lemma \ref{lem:1dim} shows that there is an EKL form $w_0(f)=\bbH$ of rank two and an EKL form of rank 3 with $w_0(f) = \bbH \oplus \langle a \rangle$ for a unit $a\in k^{\times}$, 
while Lemma \ref{lem:2dim-1} shows that there are EKL forms of rank 4 of either form $\bbH \oplus \bbH$ or $\bbH \oplus \langle a,b \rangle$ with arbitrary units $a,b\in k^{\times}$. 
We will extend this list in the next section.  
Lemma \ref{lem:2dim-1} will provide examples for the study of EKL forms over finite fields in section \ref{sec:finite-fields}. 

In the course of the calculations, we frequently encounter symmetric bilinear forms with an anti-diagonal Gram matrix. Specifically, we encounter $n\times n$ matrices $A  = (a_{ij})_{i,j=1}^n$ where $a_{ij} = a\neq 0$ when $i+j = n+1$ and $a_{ij}=0$ otherwise. The associated quadratic form of $A$ is isometric to $\frac{n}{2}\bbH$ when $n$ is even and $\frac{n-1}{2}\bbH \oplus \langle a \rangle$ when $n$ is odd. For example, any $2\times 2$ anti-diagonal matrix with $a_{1,2}=a_{2,1}=a$ is easily seen to be congruent to the diagonal matrix $\mathrm{diag}\{1, -1\}$, and thus the associated quadratic form is isometric to the hyperbolic form $\bbH$. 

\begin{lemma}
\label{lem:2dim-1}
For any constants $a,b \in k^{\times}$, consider the map $f : \bbA^2 \to \bbA^2$ given by $f_1 = xy$ and $f_2 = -ax^2 + by^2$. If $\phi : Q_0(f) \to k$ is a $k$-linear function with $\phi(E)=1$, the associated quadratic form $\beta_\phi$ is congruent to the diagonal form $\langle 1, -1, a, b\rangle$.
\end{lemma}

\begin{proof}
Use the graded reverse lexicographic ordering on $k[x,y]$ with $x>y$. The $S$-polynomial of $f_1$ and $f_2$ is $S(f_1, f_2) =\frac{b}{a}y^3$, which yields a Gr{\"o}bner basis $(xy, -ax^2+by^2, y^3)$ for the ideal $(f_1, f_2) \subseteq k[x,y]$. As both $x^3$ and $y^3$ are in the ideal $(f_1, f_2)$, the quotient ring $k[x,y]/(f_1,f_2)$ is a local ring with maximal ideal $(x_1,x_2)$. By basic Gr{\"o}bner basis theory, the ring $Q_0(f) = k[x,y]/(f_1,f_2)$ has $\{1,y,y^2,x\}$ as a $k$-basis. One calculates $E = b y^2$, and so we may take $\phi(y^2) = b^{-1}$ and declare $\phi$ to vanish on $1$, $x$, and $y$. Using the ordered basis $\{1,y,y^2,x\}$, the matrix for the symmetric bilinear form $\beta_\phi$ is 
\begin{equation}
    \begin{pmatrix}
    0 & 0 & b^{-1} & 0 \\ 
    0 & b^{-1} & 0 & 0 \\
    b^{-1} & 0 & 0 & 0 \\
    0 & 0 & 0 & a^{-1}
    \end{pmatrix},
\end{equation}
which is congruent to the diagonal form $\langle 1, -1, a, b \rangle$.
\end{proof}

\begin{lemma}
\label{lem:1dim}
For any unit $a \in k^{\times}$ let $f : \bbA^1 \to \bbA^1$ be the map given by $x \mapsto ax^n$ with $n \geq 2$. Then the EKL class of $f$ is $w_0(f) = \frac{n}{2} \bbH$ when $n$ is even and $w_0(f) = \frac{n-1}{2}\bbH \oplus \langle a \rangle$ when $n$ is odd.
\end{lemma}

\begin{proof}
The local ring $Q_0(f) \cong k[x]/(x^n)$ has $\{1,x, \ldots, x^{n-1}\}$ as a $k$-basis and $E=ax^{n-1}$. Define $\phi$ by $\phi(x^{n-1}) = a^{-1}$ and $\phi(x^i) = 0$ for $i < n-1$. The matrix of the bilinear form $\beta_\phi$ is easily seen to be the anti-diagonal matrix
\begin{equation}
    \begin{pmatrix}
    0 & \cdots & 0 & a^{-1} \\ 
    0 & \cdots & a^{-1} & 0 \\
    \vdots & & \vdots \\
    a^{-1} & \cdots & 0 & 0
    \end{pmatrix},
\end{equation}
which is congruent to the diagonal form claimed in the lemma. 
\end{proof}

\begin{remark}
Lemma \ref{lem:1dim} gives a map $f : \bbA^1 \to \bbA^1$ in one variable, say $f(x_1) \in k[x_1]$, that yields the EKL form $\frac{n}{2}\bbH$ or $\frac{n-1}{2}\bbH \oplus \langle a \rangle$. We can extend $f$ to a map $g : \bbA^n \to \bbA^n$ for any $n\geq 2$ as follows: take $g_1 = f(x_1)$ and $g_i = x_i$ for all $2\leq i\leq n$. It is straightforward to verify that $w_0(f) \cong w_0(g)$, so although $w_0(g)$ arises from a map of $n$-space, it is effectively a one dimensional example. We give a general procedure for describing when an EKL form can be realized as an EKL form using fewer variables in Theorem \ref{thm:dimension_reduction}. 
\end{remark}

The next lemma is a generalization of Lemma \ref{lem:2dim-1}.

\begin{lemma}
\label{lem:2dim-2}
Let $n$ be an even natural number with $n\geq 2$ and choose units $a, b \in k^{\times}$. The map $f : \bbA^2 \to \bbA^2$ given by $f_1 = xy$, $f_2 = ay^n - b x^2$ has EKL class $w_0(f) = \frac{n}{2}\bbH \oplus \langle a, b \rangle$. 
\end{lemma}

\begin{proof}
We give a proof for $n\geq 4$ as Lemma \ref{lem:2dim-1} handles the case $n=2$. A Gr{\"o}bner basis is determined by calculating $S$-polynomials, and we must only add $S(f_1, f_2) = \frac{b}{a}x^3$ to the set $\{f_1, f_2\}$ to get a Gr{\"o}bner basis of the ideal $(f_1,f_2)$. We have $Q_0(f) \cong k[x,y]/(f_1, f_2)$ as the latter ring is a local ring, which is easily seen because both $x^3$ and $y^{n+1}$ are in the ideal $(f_1, f_2)$. A $k$-basis for $Q_0(f)$ is then given by $\{1, x, x^2, y, y^2, \ldots, y^{n-1} \}$. We calculate $E = ay^n \equiv bx^2$ in $Q_0(f)$. Hence we define $\phi$ so that $\phi(x^2) = b^{-1}$  and require $\phi$ to vanish on the other elements in our chosen basis. The resulting matrix of the bilinear form $\beta_\phi$ consists of two blocks: a $3\times 3$ block that is anti-diagonal with each anti-diagonal term equal to $b^{-1}$ and an $n-1 \times n-1$ block that is anti-diagonal with each anti-diagonal entry equal to $a^{-1}$. It is now straightforward to see that after diagonalizing this matrix, the resulting quadratic form is the one claimed in the lemma. 
\end{proof}

\begin{lemma}
\label{lem:n-m-powers}
For any units $a, b \in k^{\times}$, the map $f : \bbA^2 \to \bbA^2$ given by $f_1 = xy$ and $f_2 = - bx^n + a y^m$ with $n \geq m \geq 2$ and both $n$ and $m$ even has EKL class $w_0(f) = \frac{m}{2}\bbH \oplus \frac{n-2}{2}\bbH \oplus \langle a, b \rangle$.
\end{lemma}

\begin{proof}
The same line of argument used in the proofs of Lemmas \ref{lem:2dim-1} and \ref{lem:2dim-2} yields the result. 
\end{proof}

\begin{lemma}
\label{lem:Nth-power}
Let $f : \bbA^n \to \bbA^n$ be a map with an isolated zero at $0$, and suppose that $\dim_k Q_0(f) = N \geq 1$. Then the ideal $(f_1,..., f_n)$ in $P_0=k[x_1,...,x_n]_{(x_1,...,x_n)}$ contains $(x_1,...,x_n)^N$. 
The one variable case shows that $N=M$ is the smallest exponent that satisfies $(x_1,...,x_n)^M \subseteq (f_1\ldots,f_n)$ in general. 
\end{lemma}

\begin{proof}
Let $x^{\alpha} = x_{i_1}x_{i_2}\cdots x_{i_N}$ be a monomial of degree $N$, where the indices may be repeated and $1\leq i_j \leq n$. The set $\{ 1, x_{i_1}, x_{i_1}x_{i_2}, \ldots, x^{\alpha} \}$ must be linearly dependent in $Q_0(f)$, as it contains $N+1$ elements. Hence there is some linear relation $\sum_{j=0}^N a_j x_{i_1}\cdots x_{i_j} = 0$ in $Q_0(f)$. Let $m$ be the smallest index for which $a_m$ is nonzero. We must have $m \geq 1$ as $\dim Q_0(f) \geq 1$. Then 
\begin{align}
x_{i_1}\cdots x_{i_m}\left(\sum_{j= m}^N a_j \frac{x^{\alpha}}{x_{i_1}\cdots x_{i_m}}\right) = 0
\end{align}
implies that $x_{i_1}\cdots x_{i_m} = 0 $ in $Q_0(f)$, as the other factor is a unit in the ring $k[x_1,...,x_n]_{(x_1,...,x_n)}$. The result follows from this observation. 
\end{proof}

\begin{lemma}
\label{lem:Nth_power_diff}
Let $f: \bbA^n \to \bbA^n$ be a map with an isolated zero at $0$ and suppose $\dim_k Q_0(f) = N$. If the map $g : \bbA^n \to \bbA^n$ satisfies $f_i - g_i \in (x_1,\ldots, x_n)^{N+1}$ in $P_0$ for all $i$, then the ideals $(f_1,...,f_n)$ and $(g_1,...,g_n)$ are equal in $P_0$ and the EKL classes of $f$ and $g$ agree, i.e., $w_0(f) = w_0(g)$. 
\end{lemma}

\begin{proof}
This result is essentially \cite[Lemma 17]{KassWickelgren}. To summarize their argument, the ideal $(f_1,...,f_n)$ is shown to be equal  to $(g_1,...,g_n)$ by using Nakayama's lemma; hence $Q_0(f)=Q_0(g)$. They then verify that the distinguished socle elements $E(f)$ and $E(g)$ are equal as follows. Write $g_i = \sum a_{ij}(g) x_j$ so that $E(g) = \det (a_{ij}(g))$, then the assumption that $f_i - g_i \in (x_1,\ldots x_n)^{N+1}$ implies that the matrices $(a_{ij}(f))$ and $(a_{ij}(g))$ agree modulo $(x_1,\ldots,x_n)^N $, hence $E(f)$ and $E(g)$ agree. It now follows that the EKL classes of $f$ and $g$ agree by the discussion following Definition \ref{def:EKL}.
\end{proof}

\begin{lemma}
\label{lem:linear-combination}
Let $f=(f_1,\ldots,f_n) : \bbA^n \to \bbA^n$ be a map with an isolated zero at $0$. Then for any polynomial $h \in k[x_1,\ldots, x_n]$ and any distinct indices $i$ and $j$, the polynomial maps $f$ and $f' = (f_1,\ldots, f_j + h\cdot f_i, \ldots, f_n)$ give isomorphic EKL classes.
\end{lemma}

\begin{proof}
It is clear that the maps $f$ and $f'$ define the same local rings $Q_0(f) = Q_0(f')$, so it is only a matter of verifying that the maps $f$ and $f'$ yield the same distinguished socle element $E$. If we have $f_i = \sum a_{ij}x_j$ so that $E(f) = \det (a_{ij})$, then $E(f')$ is determined by first performing the elementary row operation of adding $h$ times the $i$th row to the $j$th row to the matrix $(a_{ij})$ and then taking the determinant. The elementary row operation does not affect the determinant, so the distinguished socle elements of $f$ and $f'$ are equal. Thus the EKL classes for $f$ and $f'$ agree. 
\end{proof}

\begin{lemma}
\label{lem:linear-composition}
Let \( f = (f_1, \ldots, f_n) \colon \bbA^n \to \bbA^n \) be a map with an isolated zero at \(0\) and let \( A \colon \bbA^n \to \bbA^n \) be a linear isomorphism. Then \(Q_0(f) \cong Q_0(A\circ f) \) and the EKL classes of \(f\) and \(A \circ f\) differ by multiplication by $\det(A)$. 

\end{lemma}

\begin{proof} 
Express $A$ as \( A = (a_{11}x_1 + \ldots + a_{1n}x_n, \ldots, a_{n1}x_1 + \ldots + a_{nn}x_n)\) with $a_{ij}\in k$ so that 
 \begin{equation*}
 A \circ f = \left( \sum_{\ell=1}^n a_{1\ell}f_\ell, \ldots, \sum_{\ell=1}^n a_{n\ell}f_\ell \right).
 \end{equation*}
It is clear that the ideals $(f_1,...,f_n)$ and  $\left( \sum_{\ell=1}^n a_{1\ell}f_\ell, \ldots, \sum_{\ell=1}^n a_{n\ell}f_\ell \right)$ agree in $P$; hence the local rings $Q_0(f)$ and $Q_0(A\circ f)$ are equal. We now verify that the distinguished socle elements of $f$ and $A\circ f$ at $0$ differ by $\det(A)$: 

Let $b_{ij}\in P$ be polynomials such that $f_i(x) = \sum_{j=1}^n b_{ij}x_j$ for each $i$.  
Composing $f$ with $A$ then yields 
 \begin{equation*}
 (A \circ f)_i = \sum_{\ell=1}^n a_{i\ell}f_\ell = \sum_{\ell=1}^n a_{i\ell} \sum_{j=1}^n b_{\ell j}x_j = \sum_{j=1}^n \left( \sum_{\ell=1}^n a_{i\ell}b_{\ell j} \right) x_j.
 \end{equation*}
Since we have \( E(f) = \det(b_{ij}) \) by definition of $E$, it follows that \( E(A \circ f) = \det(A \cdot (b_{ij})) = \det(A) E(f). \) Note that \( \det(A) \in k^{\times} \) is a unit. 
Now, given a \(k\)-linear function \( \phi \colon Q_0(f) \to k \) so that \( \phi(E(f)) = 1 \), define \( \Tilde{\phi} \colon Q_0(A \circ f) \to k \) by \( g \mapsto \det(A)^{-1} \phi(g) \). This is \(k\)-linear and maps \( E(A \circ f) \) to 1, so we can use it to compute the EKL-class of \( A \circ f \), and we get \( w_0(A \circ f) = \det(A)^{-1} w_0(f) \).
\end{proof}

We are now ready to prove a key technical result that will allow us to reduce dimensions in the following sections. 

\begin{theorem}
\label{thm:dimension_reduction}
Suppose $f : \bbA^n \to \bbA^n$ is a map with an isolated zero at $0$ and suppose that the rank of $Q_0(f)$ is $N\geq 1$. If the ideal $(f_1,...,f_n)$ is not contained in $(x_1,...,x_n)^2\subseteq k[x_1,...,x_n]$, then we can eliminate a variable in the description of $Q_0(f)$. That is, there is a map $g : \bbA^{n-1} \to \bbA^{n-1}$ with isolated zero at $0$ and $Q_0(g)\cong Q_0(f)$. Furthermore, the EKL classes of $f$ and $g$ differ only by multiplication by a unit. 
\end{theorem}

\begin{proof}
Recall that we write $P$ for the polynomial ring $k[x_1,...,x_n]$, $\frakm$ for the maximal ideal $(x_1,...,x_n)$, and $P_0$ for the local ring $P$ localized at $\frakm$. We study modifications of the generators of the ideal $I = (f_1,...,f_n)$ until we are able to eliminate a variable. We can eliminate all terms of degree greater than  $N$ in the polynomials $f_1,...,f_n$ without modifying the ideal or the EKL class by Lemma \ref{lem:Nth_power_diff}. Assume this is done, and write $f_i = \sum_{j=1}^{N}f_{ij}$ where $f_{ij}$ is a homogeneous polynomial of degree $j$. 

Without loss of generality, we may assume $f_{11}$ is non-zero, as we have assumed that some $f_i$ has a non-zero linear term. Find a linear isomorphism $A : \bbA^n \to \bbA^n$ that will transform $f_{11}$ into $x_1$. The EKL classes of $f$ and $A \circ f$ differ only by multiplication of a unit by Lemma \ref{lem:linear-composition}, so we may assume $f_{11}=x_1$ from now on. 

By adding multiples of $f_1$ to the other polynomials $f_2,...,f_n$ and then removing all terms of degree greater than $N$, we may eliminate all of the monomials that are divisible by $x_1$ from $f_2,...,f_n$. Write $g_2,...,g_n$ for the resulting polynomials in $k[x_2,...,x_n]$ obtained from this procedure. It follows that $(f_1,...,f_n) = (f_1,g_2,...,g_n)$ as ideals in $P_0$ and also the EKL classes agree by Lemmas \ref{lem:linear-combination} and \ref{lem:Nth_power_diff}. 

We now produce a polynomial $h \in k[x_2,...,x_n]$ that satisfies $x_1 - h \in I$. We do this by showing that every monomial $x^{\alpha}$ that is divisible by $x_1$ admits a polynomial $h_{\alpha}\in k[x_2,...,x_n]$ for which $x^{\alpha} - h_{\alpha} \in I$. We use a downward induction argument, observing that the claim is true when the degree of $x^{\alpha}$ is at least $N$ by taking $h_{\alpha} = 0$ as $\frakm^N \subseteq I$. 
 
Assume the result is now true for all monomials $x^{\alpha}$ of total degree $i+1$ that are divisible by $x_1$ and consider $x^{\alpha}$ of total degree $i$ that is divisible by $x_1$. We can then consider the member of $I$
\begin{equation*}
    \frac{x^{\alpha}}{x_1}\cdot f_1 = \frac{x^{\alpha}}{x_1}\cdot\left(\sum_{j=1}^N f_{1j}\right) =\frac{x^{\alpha}}{x_1}f_{1,N} + \cdots + x^{\alpha}.
\end{equation*}
All of the terms except $x^\alpha$ have degree at least $i+1$. Hence the induction hypothesis may be used to replace all of the terms divisible by $x_1$ with a polynomial in $k[x_2,...,x_n]$. This then yields an expression $x_{\alpha}  - h_{\alpha} \in I$ where $h_\alpha \in k[x_2,...,x_n]$. 

The ideals $(f_1,f_2,...,f_n)$ and $(x_1-h,g_2,...,g_n)$ agree in the local ring $P_0=k[x_1,...,x_n]_{\frakm}$\footnote{As $x_1 - h \in (f_1,...,f_n)$, there is an equation $x_1-h = \sum\frac{a_i}{b_i}f_i$, where $a_i,b_i \in P$ and the $b_i$ are units. Multiply through by a common denominator; upon relabeling coefficients, we get an equation $B(x_1 - h) = \sum a_i f_i$ in $P$. Modulo $(x_2,...,x_n)$ the equation is equivalent to $Bx_1 = a_1f_1$. Since $B$ is a unit, $Bx_1$ has a non-zero linear term. Thus $a_1f_1$ must have a non-zero linear term too, and this can happen only if $a_1$ has a non-zero constant term. Thus $a_1$ is a unit in $P_0$. From this, we can finally obtain the expression $f_1 = \frac{B}{a_1}(x_1-h) - \sum_{i\neq 1} \frac{a_i}{a_1}f_i$. }. Furthermore, the maps $(f_1,...,f_n)$ and $(x_1-h,g_2,...,g_n)$ define the same distinguished socle element up to a unit in $P_0$.\footnote{It suffices to show that $(f_1,g_2,...,g_n)$ and $(x_1-h,g_2,...,g_n)$ define the same distinguished socle element up to a unit by our initial reduction step. Write out the matrices that define $E$ in both cases. Only the first row of the two matrices can be different. But the first column of both matrices has only one non-zero entry, and it is a unit because both $f_1$ and $x_1 -h$ have linear term $x_1$.} 
The evaluation map $\psi : P_0 \to k[x_2,...,x_n]_{(x_2,...,x_n)}$ defined by $\phi(x_1) = g_1$ and $\phi(x_i) = x_i$ otherwise induces an isomorphism on local rings $\phi : Q_0(f) \to Q_0(g)$, where $g=(g_2,...,g_n)$.  It is now straightforward to see that the distinguished socle element $E(x_1-h,g_2,...,g_n)$ maps under $\phi$ to $u\cdot E(g_2,...,g_n)$ for some unit $u$. Thus the EKL classes agree up to a unit and we are done. 
\end{proof}

\begin{remark}
The previous theorem is inspired by the proof of \cite[Lemma 5.7]{Mckean} and serves as a generalization of McKean's result to higher dimensions. We thank Sabrina Pauli for bringing McKean's paper to our attention.
\end{remark}

\begin{lemma}
\label{lem:intersection-multiplicity}
Let $f : \bbA^n \to \bbA^n$ be a map with isolated zero at $0$ and assume that the ideal $(f_1,...,f_n)$ is contained in $(x_1,..,x_n)^2$. Then the dimension of $Q_0(f)$ is at least $2^n$. 
\end{lemma}

\begin{proof}
This result follows from \cite[\S 12.4]{Fulton} and the relevant result is clearly stated on \cite[Chapter 12, Notes and References, page 234]{Fulton}.
\end{proof}

The two results above show that all EKL classes with rank less than 8 can be reduced to the planar or 2-variable case. 

\section{EKL classes of fixed rank}

We will now address the question of which quadratic forms are representable as EKL forms. This question depends on the base field $k$, but we investigate the question for a general field $k$ below. 
We have already seen that any EKL form $w_0(f)$ with rank at least 2 necessarily has a hyperbolic summand.  
We have also seen that there are rank 4 EKL classes with exactly one hyperbolic summand and also some with two hyperbolic summands by Lemma \ref{lem:2dim-1}. 
What then is the minimal number of hyperbolic forms that an EKL class of rank $n$ must contain as a summand over a general field? The next case to analyze is when the rank of $w_0(f)$ is 5.
The following table summarizes the cases we know including our main result on rank 5 of Theorem \ref{thm:rank-5}. 

\begin{table}[htbp]
    \centering
    \begin{tabular}{c|c}
         Rank & EKL form type \\
         1 & $\langle a \rangle$ \\
         2 & $\bbH$ \\
         3 & $\bbH +\langle a \rangle$ \\
         4 & $\bbH + \langle a, b \rangle$\\
         5 & $2\bbH + \langle a \rangle$ \\
         6 & ? \\
    \end{tabular}
    \caption{This is a list of the possible EKL forms of a given rank. In the EKL form type, the constants $a$ and $b$ are arbitrary units in the ground field. In particular, in the rank $4$ case, it is possible to take $b=-a$ so that the EKL class is $2\bbH$.}
    \label{tab:min_hyperbolic}
\end{table}

\begin{theorem}
\label{thm:rank-5}
Let $f : \bbA^n \to \bbA^n$ be a map with isolated zero at 0, and suppose further that the rank of $w_0(f)$ is 5. Then $w_0(f) \cong 2\bbH \oplus \langle a \rangle $ for some unit $a$. 
\end{theorem}

\begin{proof}
By Theorem \ref{thm:dimension_reduction}, Lemma \ref{lem:intersection-multiplicity}, and our assumption that the rank of $w_0(f)$ is 5, we need only analyze the case of a function $f : \bbA^n \to \bbA^n$ where $n$ is 1 or 2. The one variable case was handled in Lemma \ref{lem:1dim}, so we assume now that $f = (f_1, f_2) : \bbA^2 \to \bbA^2$ has an isolated zero at 0, the rank of $w_0(f)$ is 5, and $(f_1, f_2) \subseteq \frakm^2$.

The local ring $Q_0(f)$ is isomorphic to the ring $P/((f_1, f_2) + \frakm^5)$, and so we may study its structure by studying the structure of a Gr{\"o}bner basis for the ideal $(f_1, f_2) + \frakm^5$. There are only a few possibilities of the leading terms in a Gr{\"o}bner basis in the 2 variable case that give a dimension of 5. We list them in Figure \ref{fig:groebner_1}. 

\begin{figure}[htbp]
    \centering
\begin{tikzpicture}[scale=.9]
\foreach \n in {0,...,5}{\draw (0,\n) -- (2,\n);}
\foreach \n in {0,...,2}{\draw (\n, 0) -- (\n, 5);}
\foreach \n in {0,...,4}{\filldraw [black] (0,\n) circle (4pt);}
\end{tikzpicture} \qquad
\begin{tikzpicture}[scale=.9]
\foreach \n in {0,...,5}{\draw (0,\n) -- (2,\n);}
\foreach \n in {0,...,2}{\draw (\n, 0) -- (\n, 5);}
\foreach \n in {0,...,3}{\filldraw [black] (0,\n) circle (4pt);}
\filldraw[black] (1,0) circle (4pt);
\end{tikzpicture} \qquad
\begin{tikzpicture}[scale=.9]
\foreach \n in {0,...,5}{\draw (0,\n) -- (2,\n);}
\foreach \n in {0,...,2}{\draw (\n, 0) -- (\n, 5);}
\foreach \n in {0,...,2}{\filldraw [black] (0,\n) circle (4pt);}
\filldraw[black] (1,0) circle (4pt);
\filldraw[black] (1,1) circle (4pt);
\end{tikzpicture}\qquad
\begin{tikzpicture}[scale=.9]
\foreach \n in {0,...,5}{\draw (0,\n) -- (3,\n);}
\foreach \n in {0,...,3}{\draw (\n, 0) -- (\n, 5);}
\foreach \n in {0,...,2}{\filldraw [black] (0,\n) circle (4pt);}
\filldraw[black] (1,0) circle (4pt);
\filldraw[black] (2,0) circle (4pt);
\end{tikzpicture}

\vspace{1cm}

\begin{tikzpicture}[scale=.9]
\foreach \n in {0,...,3}{\draw (0,\n) -- (3,\n);}
\foreach \n in {0,...,3}{\draw (\n, 0) -- (\n, 3);}
\foreach \n in {0,...,1}{\filldraw [black] (0,\n) circle (4pt);}
\filldraw[black] (1,0) circle (4pt);
\filldraw[black] (1,1) circle (4pt);
\filldraw[black] (2,0) circle (4pt);
\end{tikzpicture}\qquad
\begin{tikzpicture}[scale=.9]
\foreach \n in {0,...,3}{\draw (0,\n) -- (3,\n);}
\foreach \n in {0,...,3}{\draw (\n, 0) -- (\n, 3);}
\foreach \n in {0,...,1}{\filldraw [black] (0,\n) circle (4pt);}
\filldraw[black] (1,0) circle (4pt);
\filldraw[black] (2,0) circle (4pt);
\filldraw[black] (3,0) circle (4pt);
\end{tikzpicture}\qquad
\begin{tikzpicture}[scale=.9]
\foreach \n in {0,...,3}{\draw (0,\n) -- (4,\n);}
\foreach \n in {0,...,4}{\draw (\n, 0) -- (\n, 3);}
\foreach \n in {0,...,4}{\filldraw [black] (\n,0) circle (4pt);}
\end{tikzpicture}
    \caption{From left to right, top to bottom, we list the leading terms of generators for the Gr{\"o}bner basis. (1) $y^5$, $x$; (2) $y^4$, $xy$, $x^2$; (3) $y^3$, $xy^2$, $x^2$; (4) $y^3$, $xy$, $x^3$; (5) $y^2$, $x^2y$, $x^3$; (6) $y^2$, $xy$, $x^4$; (7) $y$, $x^5$. }
    \label{fig:groebner_1}
\end{figure}
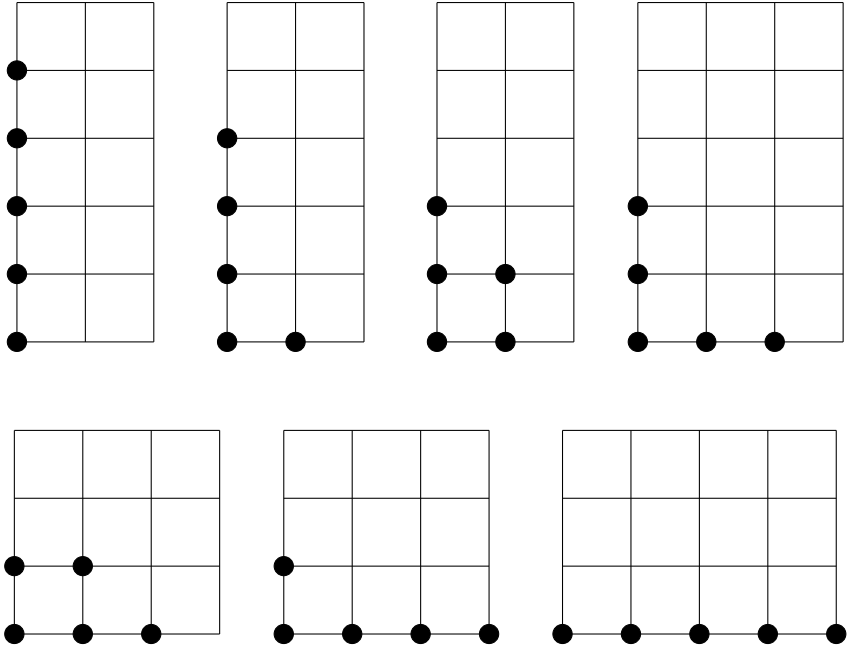

The crux of the argument is to find a subspace $\Span\{h_1,h_2\} \subseteq Q_0(f)$ that is totally isotropic, which then guarantees that $2\bbH$ is a summand of $w_0(f)$ (see Scharlau for a more general result \cite[Theorem 4.5.]{Scharlau}). In our case, the two-dimensional subspace $\Span \{h_1,h_2\}$ is totally isotropic if $\phi(h_1\cdot h_1) = \phi(h_1 \cdot h_2) = \phi (h_2 \cdot h_2) = 0$. We manage to do one better and find $\{h_1,h_2\}$ which satisfy
\begin{align}
\label{eq:isotropic_hs}
h_1^2 = h_2^2 = h_1h_2 = 0 ~\text{in}~ Q_0(f).
\end{align}

The proof now proceeds by a case-by-case analysis in Lemma \ref{lem:cases}.
\end{proof}

\begin{lemma}
\label{lem:cases}
We consolidate a case-by-case analysis of the ring structure of the local rings with Gr{\"o}bner basis of type (1)--(7) found in figure \ref{fig:groebner_1}. In particular, the local ring $Q_0(f)$ contains a linearly independent subset $\{h_1,h_2\}$ which satisfies $h_1^2 = h_2^2 = h_1h_2 = 0$. 
\end{lemma}

\begin{proof}
As described above, %
we only need to consider maps $f: \bbA^2 \to \bbA^2$ with $(f_1,f_2) \subseteq \frakm^2$. Express the polynomials $f_1$ and $f_2$ as a sum of homogeneous polynomials as follows: $f_1 = f_{1,5} + \cdots + f_{1,2}$ and $f_2 = f_{2,5} + \cdots + f_{2,2}$. We are assuming that the intersection multiplicity of $V(f_1)$ and $V(f_2)$ at $0$ is 5. This can only happen when the projective tangent cones meet in exactly one point.\footnote{If the projective tangent cones were disjoint and $f_{1,2}$ and $f_{2,2}$ are non-zero, then the intersection multiplicity is exactly 4. If the projective tangent cones are identical, then $f_{1,2}$ and $f_{2,2}$ differ only by a scalar, and thus we can eliminate $f_{2,2}$ by adding some multiple of $f_1$ to $f_2$ without changing the intersection multiplicity. But then the intersection multiplicity is at least 6 in this case. Thus the only possibility is that $f_{1,2}$ and $f_{2,2}$ are both non-zero, split over $k$, and share a common factor.} This necessitates that both $f_{1,2}$ and $f_{2,2}$ split as a product of lines defined over the ground field and share exactly one factor.  We can thus compose with an invertible linear transformation to modify $f_{1,2}$ to $xy$ and $f_{2,2}$ to $x(\alpha x+\beta y)$. This modification will only affect the EKL class by scalar multiplication by a unit by Lemma \ref{lem:linear-composition}. 

We now assume that $f_{1,2} = xy$ and $f_{2,2} = \alpha x^2 + \beta xy$. But by adding $-\beta f_1$ to $f_2$ we do not change the EKL class and obtain $f_{1,2} = xy$ and $f_{2,2} = \alpha x^2$. It must be that $\alpha \neq 0$ for otherwise the rank of the EKL class would be at least 6. 

We now proceed to our case-by-case analysis with the reduction above performed. Recall that we use the graded reverse lexicographic ordering on monomials $x^iy^j$ with $x > y$. 

First note that cases (1) and (7) are not permitted as the Gr{\"o}bner basis structure implies $I = (f_1, f_2) + \frakm^5$ is not contained in $\frakm^2$. 

Case (2): We know that the ideal $(f_1, f_2) + \frakm^5$ has a Gr{\"o}bner basis of the form 
\begin{align*}
    g_1 &= y^4 + ax^3 + bx^2y + cxy^2 + dy^3 + ex^2 + fxy + gy^2 \\
    g_2 &= xy + hy^2\\
    g_3 &= x^2 +ixy +j y^2.
\end{align*}
Hence the ideal $I = (f_1,f_2) \subseteq P_0$ is equal to $(g_1,g_2,g_3)$, which we now investigate. The dimension of $I/\frakm I$ must be at most 2 as the ideal $I$ is generated by $(f_1,f_2)$ (this is a consequence of Nakayama's Lemma). Thus the image of $I/\frakm I$ in $\frakm^2/\frakm^3$ is at most 2 dimensional, and our assumptions show that $\{x^2, xy\}$ is in the image. So there can be no $y^2$ term in $g_2$ or $g_3$. Thus the generators $(g_1,g_2,g_3)$ can be modified to give a new generating set of $I$ of the form 
\begin{align*}
    g_1 &= y^4 + ay^3  \\
    g_2 &= xy \\
    g_3 &= x^2.
\end{align*}
for some new constant $a$. Now observe that $yg_1 \in I$ too, that is, $y^5 + ay^4 \in I$ and we also know $y^5 \in I$. In any event, $y^4 \in I$. Thus we conclude $\frakm^4 \subseteq I$. Hence we have $\{x, y^2\}$ a linearly independent subset of $Q_0(f)$ satisfying \eqref{eq:isotropic_hs}, i.e., that it is totally isotropic. 

Case (3): The ideal $(f_1, f_2) + \frakm^5$ has a Gr{\"o}bner basis of the form 
\begin{align*}
    g_1 &= y^3 + ax^2 + bxy + cy^2 \\
    g_2 &= xy^2 + dy^3 + ex^2 + fxy + gy^2 \\
    g_3 &= x^2 + hxy + i y^2.
\end{align*}
Thus $I = (f_1,f_2) = (g_1,g_2,g_3)$ in $P_0$. By our assumption that the image of $I/\frakm I$ in $\frakm^2/\frakm^3$ is generated by $\{x^2, xy\}$, we can simplify the Gr{\"o}bner basis to
\begin{align*}
    g_1 &= y^3  + axy  \\
    g_2 &= xy^2 + bxy  \\
    g_3 &= x^2 + cxy.
\end{align*}
for some constants $a,b,c$. Observe now that $g_2$ factors as $xy(y+b)$. If $b\neq 0$, then $(y+b)$ is invertible in $Q_0(f)$, hence we may conclude that $\{xy,x^2,y^3\}$ generates $I$, from which it follows that $\dim_k Q_0(f) < 5$. Thus we need only consider the case when $b = 0$. Now as $yg_1 = y^4 + axy^2 \in I$ and $xy^2 \in I$ it follows that $y^4 \in I$. But now, $x^2y^2 \in I$ and the set $\{y^2, xy\}$ satisfies \eqref{eq:isotropic_hs} and is totally isotropic in $Q_0(f)$. 

Case (4): The ideal $(f_1, f_2) + \frakm^5$ has a Gr{\"o}bner basis of the form 
\begin{align*}
    g_1 &= y^3 + ax^2 + bxy + cy^2 \\
    g_2 &= xy + dy^2 \\
    g_3 &= x^3 + ex^2y + fxy^2 + gy^3 + hx^2 + ixy + jy^2.
\end{align*}
Thus $I = (f_1, f_2) = (g_1,g_2,g_3)$ in $P_0$. Because both $xy$ and $x^2$ are in the image of the map $I/\frakm I \to \frakm^2/\frakm^3$, we  conclude that no $y^2$ terms are permitted. Thus $g_2 = xy$ and we may simplify the generators to 
\begin{align*}
    g_1 &= y^3 + ax^2  \\
    g_2 &= xy  \\
    g_3 &= x^3 + bx^2.
\end{align*}
with $a$ and $b$ redefined constants. We now observe that $yg_1 = y^4 + ax^2y$ is in the ideal, and thus so too is $y^4$. 
Observe that $g_3 = x^2(x+b) \in I$ implies that $x^3 \in I$ if $b = 0$ or $x^2 \in I$ if $b\neq 0$, as in the local ring $x+b$ is a unit. In either case, it follows that $x^4 \in I$. Thus $\frakm^4 \subseteq I$. 
Hence we have $\{x^2, y^2\}$ satisfies \eqref{eq:isotropic_hs} and is a totally isotropic subset of $Q_0(f)$.

Cases (5) and (6): In these cases, we are forced to have $y^2$ as a member of the Gr{\"o}bner basis, which is not permitted by our assumption that the image of $I/\frakm I$ in $\frakm^2/\frakm^3$ is generated by $\{xy, x^2\}$. 

\end{proof}

\section{The inequalities of Eisenbud, Levine, and Teissier}

As was mentioned in the introduction, our calculations in Theorems \ref{thm:H-summand} and \ref{thm:rank-5} can be explained by the inequalities of Eisenbud, Levine, and Teissier in \cite{ELT}. Throughout the section, $k$ is a field with characteristic different from 2. We present their results now and consider the implications for the use of EKL forms in motivic homotopy theory. 

Eisenbud and Levine show that an EKL form has $n$ hyperbolic summands if and only if the local ring $Q_0(f)$ contains an ideal $I$ of dimension $n$ which satisfies $I^2 = 0$. When such an ideal is found, the dual ideal for $I$, denoted by $I^*$, is used to produce the hyperbolic subspace $I \oplus I^*$ of the EKL form. Thus the quadratic space $Q_0(f)$ with the quadratic form $w_0(f)$ decomposes as $Q_0(f) = I \oplus I^* \oplus D$, where $D$ is an anisotropic subspace, meaning that it contains no isotropic subspaces. Note that $\dim_k I$ is the number of hyperbolic summands of the quadratic form $w_0(f)$. Eisenbud and Levine state their results by describing how large the dimension of $D$ can possibly be. 

\begin{theorem}
\label{thm:ELT}
Let $k$ be a field with characteristic different from $2$. The Eisenbud--Levine--Teissier inequality states that if $f : \bbA^n_k \to \bbA^n_k$ is a polynomial map which has an isolated zero at $0$ so that $\dim_k Q_0(f) \geq 1$, then the dimension of the anisotropic part of the quadratic space $(Q_0(f), w_0(f))$ is bounded by:
\begin{align}
    \dim_k D & \leq \dim_k Q_0(f)^{1-1/n} \\
    \dim_k D & \leq \frac{\dim_k Q_0(f)}{2}.
\end{align}
The first inequality is tighter, but the latter has the advantage of being independent of $n$.
\end{theorem}

\begin{corollary}
For a map $f : \bbA^n_k \to \bbA^n_k$ with an isolated zero at $0$, it is
an immediate consequence of the ELT inequality that the number of hyperbolic summands in $(Q_0(f), w_0(f))$ must at least be: 
\begin{align*}
    \dim_k I  & \geq \frac{\dim_k Q_0(f) - \dim_k Q_0(f)^{1-1/n} }{2} \\ 
    \dim_k I & \geq \frac{\dim_k Q_0(f)}{4}.
\end{align*}
\end{corollary}

With these inequalities, we can extend the chart in Table \ref{tab:min_hyperbolic} with the values in Table \ref{tab:min_hyperbolic_2}. Unfortunately, the case of rank 8 is unclear. It is possible that there is a 3 variable example with EKL form $2\bbH \oplus \langle a, b ,c ,d\rangle$. In the two variable case, the ELT inequality says that the only possible EKL forms are $3\bbH \oplus \langle a, b\rangle$. 

\begin{table}[htbp]
    \centering
    \begin{tabular}{c|c}
         Rank & EKL form type \\
         6 & $2\bbH \oplus \langle a, b\rangle $\\
         7 & $3\bbH \oplus \langle a \rangle$ \\
         8 & ?
    \end{tabular}
    \caption{This is a list of the possible EKL forms of a given rank. In the EKL form type, the constants $a$ and $b$ are arbitrary units in the ground field.}
    \label{tab:min_hyperbolic_2}
\end{table}

It is unclear what exactly happens in rank 9. We have found an example over $\bbR$ with EKL form $3\bbH \oplus \langle 1,1,1\rangle$, see Example \ref{ex:rank_9}. But does this mean that over a general field $3\bbH \oplus \langle a, b ,c \rangle$ is also always realizable as an EKL form, or only some proper subset of these? E.g., over $\bbQ$ we may indeed have $3\bbH \oplus \langle 3,3,3\rangle$ representable as an EKL form, but that does not mean that $3\bbH \oplus \langle 3,5,7\rangle$ is representable too. Can we fully describe the set of quadratic forms representable by EKL forms at 0? 

\section{Computational results}
The process of computing the EKL class of a map $f: \bbA^n\to \bbA^n$ can be automated and performed by a computer. Sabrina Pauli \cite{Pauli} uses Macaulay2 to perform such computations over finite fields, for example. We have written a simple python program that calculates the EKL class of a map, which is available on Github \cite{Wilson-local}. In this section, we provide some experimental results that give upper bounds for the minimal number of hyperbolic summands that must appear in an EKL form of rank $n$. Table \ref{tab:min_hyperbolic} gives the complete information of EKL forms up to rank 5, and in this section we will see how complicated it may be to extend this table in general. 

\begin{definition}
Let $m$ and $n$ be positive natural numbers and $k$ a field. Write $M(m,n;k)$ for the minimum number of hyperbolic summands appearing in rank $m$ EKL forms $w_0(f)$ for all possible maps $f : \bbA^n \to \bbA^n$ over the field $k$, and define $M(m)$ to be the minimum of $M(m,n;k)$ over the set of all possible $n$ and $k$. We also consider the numbers $N(m,n;k) = m - 2M(m,n;k)$ and $N(m) = m - 2M(m)$, which tell us the largest possible rank of a representative of $w_0(f)$ in the Witt ring, i.e., the rank of the ``Grundform'' of $w_0(f)$ \cite[page 35]{Witt}.  
\end{definition}

Our work up until now has established the values of $M(m)$ for $m \leq 7$, which are: $M(1) = 0$, $M(2) = 1$, $M(3) = 1$, $M(4) = 1$, $M(5) = 2$, $M(6) = 2$, $M(7) = 3$.

Morel's $\bbA^1$-degree gives a ring map $\deg^{\bbA^1} : [\bbP^n/\bbP^{n-1}, \bbP^n/\bbP^{n-1}] \to \GW(k)$ by his work in \cite[Theorem 6.3.3 and p.\ 427]{Morel03}.
Because Morel's $\bbA^1$-degree and EKL form agree in $\GW(k)$ for the local degree at a $k$-point, it follows that EKL forms must satisfy a chain rule. 
An algebraic proof of the chain rule for EKL forms was established by Knight, Swaminathan, and Tseng in \cite[Theorem 13]{KST_published}. 
We believe that their argument to establish the chain rule is correct, however, we noticed a few inaccuracies in the written proof in \cite[Theorem 13, page 80]{KST_published}. 
We provide a proof of the chain rule following their argument for completeness. We do not claim any originality. 

\begin{theorem}
\label{thm:chainrule}
Let $f,g \colon \bbA^n \to \bbA^n$ be maps with isolated zeros at the origin. 
The EKL form of the composition $f \circ g$ is the product of the EKL forms of $f$ and $g$, i.e.,
\begin{align*}
w_0(f\circ g) = w_0(f) \cdot w_0(g)~\text{in}~\GW(k).     
\end{align*}
\end{theorem}
\begin{proof}
We follow the proof by Knight, Swaminathan and Tseng with a minor correction. 
The idea is to add further variables and to compose with appropriate linear transformations to make $f$ and $g$ act on separate variables. 
So let $\tilde{f},\tilde{g} : \bbA^n\times \bbA^n \to \bbA^n \times \bbA^n$ be defined by sending $(x,y)$ to $(f(x),y)$ and $(g(x),y)$, respectively, where we write $(x,y)$ for $(x_1,\ldots,x_n,y_1,\ldots,y_n)$. 
It is straightforward to compute that the EKL class of a product of two maps $\bbA^n \to \bbA^n$ is the product of the EKL classes. 
Since $\tilde{f}\circ \tilde{g}$ equals $f \circ g$ on the first component and is the identity on the second component, we get 
\[
w_0(\tilde{f}\circ \tilde{g})=w_0(f\circ g)
\]
in $\GW(k)$. 
Hence it suffices to show $w_0(\tilde{f}\circ \tilde{g}) = w_0(f)\cdot w_0(g)$. 
The key tool is \cite[Lemma 12]{KST_published} which states that for maps $f$ and $g$ as in the theorem and a unipotent linear transformation $L: \bbA^n \to \bbA^n$ we have 
\begin{align}
\label{eq:KSTlemma12}
w_0(f \circ L \circ g) = w_0(f \circ g).
\end{align}
Note that this statement agrees with Lemma \ref{lem:linear-composition} for the case that $A=L$ is unipotent (and thus has determinant equal 1) and $f$ being the identity. 

Let $I_n$ denote the $n\times n$-identity matrix. 
Consider the following three unipotent matrices: 
\begin{align*}
L_1 = \begin{pmatrix}
I_n & 0 \\
-I_n & I_n
\end{pmatrix}, \quad
L_2 = \begin{pmatrix}
I_n & I_n \\
0 & I_n
\end{pmatrix}, \quad
L_3 = \begin{pmatrix}
I_n & 0 \\
-I_n & I_n
\end{pmatrix}. 
\end{align*}
Then it follows from \eqref{eq:KSTlemma12}, i.e., \cite[Lemma 12]{KST_published}, that the following two compositions have the same EKL forms: 
\begin{align*}
\tilde{f} \circ \tilde{g} : (x,y) \mapsto (f(g(x)),y) ~\text{and}~
\tilde{f} \circ L_1 \circ \tilde{g} : (x,y) \mapsto (f(g(x)),-g(x) + y). 
\end{align*}
Now we apply \cite[Lemma 12]{KST_published} to the maps $\tilde{f}$ and $(L_1\circ \tilde{g})$ and the unipotent transformation $L_2$. 
This shows that the EKL forms of the maps 
\begin{align*}
\tilde{f} \circ L_1 \circ \tilde{g} : (x,y)  \mapsto (f(g(x)),-g(x) + y) \\
\tilde{f} \circ L_2 \circ (L_1 \circ \tilde{g}) : (x,y) \mapsto (f(y),-g(x)+y) 
\end{align*}
are equal. 
Next we apply \cite[Lemma 12]{KST_published} to the maps $\tilde{f}$ and $(L_2\circ L_1 \circ \tilde{g})$ and the unipotent transformation $L_3$. 
This shows that the EKL forms of the maps 
\begin{align*}
\tilde{f} \circ L_2 \circ (L_1 \circ \tilde{g}) : (x,y) \mapsto (f(y),-g(x)+y) \\
\tilde{f} \circ L_3 \circ (L_2 \circ L_1 \circ \tilde{g}) : (x,y) \mapsto (f(y),-g(x)) 
\end{align*}
are equal as well. 
Denote the map $(x,y)\mapsto (f(y),-g(x))$ by $f\times (-g)$. 
We have shown that $w_0(\tilde{f}\circ \tilde{g})$ equals $w_0(f\times (-g))$.

Now we compose $f\times (-g)$ with the matrix 
$A=\begin{pmatrix}
0 & -I_n \\
I_n & 0
\end{pmatrix}$ and deduce from Lemma \ref{lem:linear-composition} that $f\times (-g)$ and $A\circ (f\times (-g)) = g \times f$ have the same EKL class.\footnote{Alternatively, one can apply \cite[Lemma 12]{KST_published} in three successive steps to the composition $L_6 \circ L_5 \circ L_4 \circ (f\times (-g))=g\times f$ where $L_4$, $L_5$, $L_6$ denote the unipotent matrices 
\begin{align*}
L_4 = \begin{pmatrix}
I_n & -I_n \\
0 & I_n
\end{pmatrix}, \quad
L_5 = \begin{pmatrix}
I_n & 0 \\
I_n & I_n
\end{pmatrix}, \quad
L_6 = \begin{pmatrix}
I_n & -I_n \\
0 & I_n
\end{pmatrix} 
\end{align*}
such that their composition is $L_6 \circ L_5 \circ L_4 = A$.
}
Hence we have 
\begin{align*}
w_0(f\circ g) = w_0(\tilde{f}\circ \tilde{g}) = w_0(g \times f).
\end{align*}
To conclude, it is straightforward to check that the EKL class of  the product $g\times f$ satisfies $w_0(g \times f)=w_0(g)\cdot w_0(f)$ in $\GW(k)$. 
This concludes the proof.
\end{proof}

From this, we can obtain some simple bounds for $M(m,2;k)$ and $N(m,2;k)$. 

\begin{corollary}
The set of quadratic forms that are realizable as EKL forms at 0 over $k$ is a multiplicative submonoid of $\GW(k)$.
\end{corollary}
\begin{proof}
The claim follows immediately from the chain rule. 
\end{proof}

\begin{example}
\label{ex:rank_9}
Consider the polynomials $f = -x^3 -x^2y + 4 xy^2 + 2y^3$ and $g = 2x^3 -x^2y - 5 xy^2 - y^3$. The EKL form of the map $(f,g) : \bbA^2 \to \bbA^2$ considered over the base field $\bbQ$ is the quadratic form $\langle 6, 6, 3, -6, -6, 6, 3, -6, 3\rangle \cong 3\bbH + \langle 3,3,3 \rangle $, which was computed with the assistance of a computer. Hence for rank 9 EKL forms, we may only have 3 hyperbolic summands. Hence $M(9,2;k)\leq 3$, i.e., $N(9,2;k)\geq 3$.
\end{example}

\begin{proposition}
\label{prop:bounds}
Consider the map $f = (xy, y^2-x^2)$ studied in Lemma \ref{lem:2dim-1}. Then the iterated composition $f^{(n)}$ has EKL form 
\begin{equation*}
    w_0(f^{(n)}) = 2^{n-1}(2^n-1) \bbH + 2^n\langle 1 \rangle
\end{equation*}
In particular, $M(4^n,2;k) \leq 2^{n-1}(2^n-1)$, i.e., $N(4^n,2;k)\geq 2^n$. 

Likewise, using the map $f$ in Example \ref{ex:rank_9}, the chain rule shows 
\begin{equation*}
w_0(f^{(n)}) = \frac{9^n - 3^n}{2} \bbH + 3^n \langle 1 \rangle;
\end{equation*}
hence $N(9^n,2;k)\geq 3^n$.
\end{proposition}

\begin{proof}
The calculations follow immediately from the chain rule. 
\end{proof}

Computational evidence for small values of $n$ suggest that the bounds in Proposition \ref{prop:bounds} are in fact equalities: $N(9^n,2;k) = 3^n$ and $N(4^n,2;k) = 2^n$, suggesting the equation $N(p^{2n},2;k)=p^n$ holds. In fact, the general result of Eisenbud and Levine in \cite[Theorem 3.9 (i)]{ELT} gives the bound $N(n,2;k)\leq \sqrt{n}$, which when combined with Proposition \ref{prop:bounds} gives the results $N(9^n,2;k)=3^n$ and $N(4^n,2;k)=2^n$ that we anticipated. Furthermore, we have observed that $N(3)=N(5)=1$ and computational experiments suggest $N(7)=1$ and $N(11,2;k)=1$ too. It is reasonable to conjecture that for odd primes $p$ that $N(p,2;k)=1$, and perhaps $N(p)=1$ holds more generally. 

Eisenbud and Levine show that the bounds above are not tight when maps $f : \bbA^n \to \bbA^n$ are allowed for $n\geq 3$. In \cite[Example, page 24]{ELT} they produce an example $f : \bbA^4_\bbR \to \bbA^4_\bbR$ where $\dim Q_0(f) = 16$ where the signature of $w_0(f)$ is $6$. That is, $N(16,4;\bbR)\geq 6$ whereas our bound in the planar case is only $N(16,2;k)\geq 4$.

\section{Applications}

\subsection{Representability in motivic homotopy theory}

For any map $f : \bbA^n \to \bbA^n$ with an isolated zero at $0$, Kass and Wickelgren \cite[Definition 11]{KassWickelgren} construct a map in the motivic homotopy category $f_0 : \bbP^n/\bbP^{n-1} \to \bbP^n/\bbP^{n-1}$ that encodes the local behavior of $f$ at $0$. To obtain a local degree in motivic homotopy theory, Kass and Wickelgren apply Morel's $\bbA^1$-Brouwer degree $\deg^{\bbA^1} : [\bbP^n/\bbP^{n-1}, \bbP^n/\bbP^{n-1}] \to \GW(k)$ to $f_0$ to obtain the class of a quadratic form, which we write as $\deg^{\bbA^1}_0(f) = \deg(f_0)$. The main result of Kass and Wickelgren in \cite{KassWickelgren} is that the local degree $\deg^{\bbA^1}_0(f)$ defined using Morel's degree map is equal to the class of the EKL form of $f$ at $0$ in $\GW(k)$. Note that Morel's degree map is in fact an isomorphism when $n \geq 2$, so that algebraic results about EKL forms can be translated into statements about motivic homotopy theory.

\begin{theorem}
Suppose $q \in \GW(k)$ is represented by a quadratic form with rank at least 2 that does not represent $0$, that is, $q$ admits no non-zero vector $v$ for which $q(v)=0$. Then there is no map $f : \bbA^n \to \bbA^n$ with an isolated zero at $0$ for which $\deg^{\bbA^1}_0(f) = q$ in $\GW(k)$. 
\end{theorem}

\begin{proof}
This follows immediately from the identification of the local $\bbA^1$-Brouwer degree with the EKL form of a map by Kass and Wickelgren, combined with Theorem \ref{thm:H-summand}. 
\end{proof}

\begin{corollary}
If $-1$ is not a square in $k$, then the quadratic form $\langle 1, 1\rangle$ is not representable as a local degree map under Morel's $\bbA^1$-Brouwer degree isomorphism. 
\end{corollary}

Because Morel's degree homomorphism $\deg : [\bbP^n/\bbP^{n-1}, \bbP^n/\bbP^{n-1}] \to \GW(k)$ is an isomorphism for $n \geq 2$, we can translate our results into statements about the set of motivic homotopy classes $[\bbP^n/\bbP^{n-1}, \bbP^n/\bbP^{n-1}]$. We give an example using our result about rank 5 EKL forms. 

\begin{theorem}
\label{thm:maps-rank-5}
Let $k$ be a field and consider a map $g : \bbP^n/\bbP^{n-1} \to \bbP^n/\bbP^{n-1}$ with $n\geq 2$. If the motivic Brouwer degree of $g$ is represented by a quadratic form $q$ of rank 5, then $g$ is $\bbA^1$-homotopy equivalent to a local degree map $f_0$ for some map $f : \bbA^n \to \bbA^n$ if and only if $q$ is of the form $2\bbH + \langle a \rangle$.
\end{theorem}

\begin{proof}
This follows from the fact that Morel's degree map is an isomorphism when $n\geq 2$ \cite[Corollary 1.24]{Morel12} and our result on rank 5 EKL forms in Theorem \ref{thm:rank-5}. %
\end{proof}

Of course, similar results in rank $6$ and $7$ follow directly from the inequality of Eisenbud, Levine, and Teissier described in Theorem \ref{thm:ELT}. In particular, in ranks 1, 3, 5, and 7 the discriminant of the EKL form contains all of the information of the local degree.  

\subsection{Finite fields}
\label{sec:finite-fields}
Over a finite field with odd characteristic, the problem of identifying the isometry class of a quadratic form in $\GW(k)\cong \bbZ \oplus \bbZ/2 $ is resolved simply by considering the rank and discriminant of the quadratic form. With this simple characterization, the lemmas of section \ref{sec:computations} allow us to completely classify which quadratic forms arise as EKL classes.  

\begin{theorem}
\label{thm:fq}
Let $\bbF_q$ be a finite field of odd characteristic. Every quadratic form over $\bbF_q$ is representable as an EKL class except for the rank 2 form with discriminant not equal to $-1 \in k^{\times}/k^{\times 2}$.  
\end{theorem}

\begin{proof}
Lemma \ref{lem:1dim} shows that there is a map $f: \bbA^1 \to \bbA^1$ with EKL class $w_0(f) = \langle a \rangle$ for any unit $a$, which covers the rank 1 case. In rank 2, Theorem \ref{thm:H-summand} shows that the only EKL class is $\bbH$, which has discriminant -1. Thus when $q \equiv 3 \bmod 4$, the unit $-1$ is not a square, and so $\langle 1, 1\rangle$ is not representable as an EKL class. When $q\equiv 1 \bmod 4$, the unit $-1$ is a square class, and thus any rank 2 form with discriminant a non-square is not representable as an EKL class. 

Any quadratic form of rank $2n +1 \geq  3$ is representable as an EKL class because Lemma \ref{lem:1dim} shows that we can choose a map $f : \bbA^1 \to \bbA^1$ so that $w_0(f)$ has rank $2n+1$ and discriminant $a$ for any unit $a \in k^{\times}$. 

For a quadratic form of rank $2n \geq 4$, Lemma \ref{lem:n-m-powers} shows that we can construct a map $f : \bbA^2 \to \bbA^2$ of rank $2n$ such that $w_0(f)$ has discriminant $(-1)^{n-1}ab$ for any choice of units $a,b\in k^{\times}$. Hence we can realize both isometry classes of quadratic forms of rank $2n$  as EKL classes. 
\end{proof}

\subsection{Singularities with specified Milnor number}
\label{sec:real}

Another avenue for applications for the structure of EKL forms that we have observed is to the study of isolated hypersurface singularities, following the work of Kass and Wickelgren \cite{KassWickelgren}. In particular, our results produce restrictions on the way singularities can degenerate into nodes just by knowing the Milnor number of the singularity in the cases where the Milnor number is at most $5$. We require the base field to have characteristic different from $2$ in this section.

We refer the reader to Kass and Wickelgren \cite[\S 8]{KassWickelgren} and Pauli and Wickelgren \cite[\S6.2]{PauliWickelgren} for a more thorough discussion, but we summarize one kind of deformation here. Let $X = \{ f = 0 \}$ be a hypersurface in $\bbA^n$ and assume that $0$ is an isolated zero of its gradient, $\grad(f) : \bbA^n \to \bbA^n$. We can then look at the local degree of $\grad(f)$ at $0$ as an invariant of the singularity of $f$ at $0$. Over the field $k=\bbC$, the Milnor number of $f$ at $0$ is $\mu(f) = \dim_k Q_0(\grad(f))$. This definition may be enriched in motivic homotopy theory by defining the $\bbA^1$-Milnor number of $f$ at $0$ to be $\mu^{\bbA^1}(f) = \deg^{\bbA^1}_0(\grad(f))$. 

Kass and Wickelgren study the family of deformations of the hypersurface $f$ that take the form $\{ f(x_1,...,x_n) + \sum_i a_ix_i = t \}$, i.e., the fibers of the map $F = f + \sum_i a_i x_i : \bbA^n \to \bbA^1$. For a generic choice of $k$-point $(a_1,...,a_n) \in \bbA^n(k)$, the hypersurfaces in this family (the fibers of the map) have only nodal singularieties. Kass and Wickelgren prove in \cite[Equation (6) and Corollary 45]{KassWickelgren} that $\mu^{\bbA^1}(f)$ counts the number of nodal fibers of this family. To be more precise, if $f : \bbA^n \to \bbA^1$ is such that $\grad(f)$ is finite and separable and $0$ is the only singularity of $f$, then 
\begin{equation}
    \mu^{\bbA^1}(f) = \sum_{x \text{ node of } F} \Tr_{k(x)/k}\type(x,f)
\end{equation}
where $\type(x,f)$ is the $\bbA^1$-Milnor number of the singular point of the nodal fiber of $F$ corresponding to $x$. 

Pauli and Wickelgren \cite[\S 6.2]{PauliWickelgren} extend the previous result to more general deformations. For a hypersurface $ X = \{ f = 0 \} \subseteq \bbA^n$, Pauli and Wickelgren are able to obtain the same stability result on the bifurcation of an isolated singularity of $f$ under more general deformations, which take the form of a fiber of the map $f + t g : \bbA^n_{k[[t]]} \to \bbA^1_{k[[t]]}$, for $g\in k[x_1,...,x_n][[t]]$. We encourage the reader to see their paper for a more thorough exposition. 

Our results can now be applied by making assumptions on the Milnor number of the singularity of $f$ at $0$ and inferring restrictions on what kinds of families of nodes the singularity can degenerate into. The next theorem is the evident generalization of \cite[Example 14]{PauliWickelgren} from the cusp---which has $\bbA^1$-Milnor number $\bbH$---to the case of any singularity with Milnor number 2. 

\begin{theorem}
Assume the characteristic of $k$ is not $2$ and assume $-1$ is not a square in $k$. Consider any hypersurface $X = \{ f = 0 \}$ in $\bbA^n$ for which $\grad(f)$ has an isolated zero at $0$. If the Milnor number of $f$ at zero $\mu^{\bbA^1}(f) = \dim_k Q_0(\grad(f))$ is equal to $2$, then if the singularity degenerates into a pair of $k$-rational nodes, it must be that the nodes have type $\langle a \rangle $ and $\langle - a \rangle$ for some unit $a\in k^{\times}$. 
\end{theorem}

\begin{proof}
This follows from \cite[Theorem 4]{PauliWickelgren} and Theorem \ref{thm:H-summand}. 
\end{proof}

Over the field of real numbers, the signature of the $\bbA^1$-Milnor number of a singularity encodes topological information about the Milnor fibers of the singularity. The work of Arnold \cite{Arnold-1978} presents the relevant terminology and gives several results, which we briefly summarize here. We also suggest the reader look at van Straten and Warmt for an introduction and related results \cite{vanStraten-Warmt}.

For $f : \bbA^n_\bbR \to \bbA^1_\bbR$ a map with an isolated singularity at $0$ and $f(0)=0$, let $B(0;\epsilon)$ be a sufficiently small ball about the origin and $0<\lvert \eta \rvert $ sufficiently small, so that $\eta \in \im(f)$, in particular. The real Milnor fibers of $f$ are the sets: 
\begin{equation*}
    F_{\eta} = B(0;\epsilon)\cap f^{-1}(\eta)
\end{equation*}
The topology of such a fiber is controlled by the signature of the $\bbA^1$-Milnor number of the singularity of $f$ at $0$. In particular,  Arnold \cite[Sections 1--2]{Arnold-1978} (compare with \cite[Theorem 1.3]{vanStraten-Warmt}) calculates the reduced Euler characteristic (Euler characteristic minus $1$) of the Milnor fibers to be:
\begin{equation*}
    \textrm{signature}(\mu^{\bbA^1}_0(f))  = -\tilde{\chi}(F_{-\eta})  
         = (-1)^{n-1}\tilde{\chi}(F_{\eta})
\end{equation*}
when $\eta > 0$.

From this result, the reduced Euler characteristics of the fibers $F_{\eta}$ must be between $-\mu(f)$ and $\mu(f)$. Our characterization of EKL forms shows that this naive bound can be improved, and for small enough Milnor numbers, the reduced Euler characteristic of these Milnor fibers can be completely determined. 

\begin{theorem}
Let $f : \bbA^n_{\bbR} \to \bbA^1_{\bbR}$ be a polynomial map with isolated singularity at $0$ with Milnor number $m \geq 2$. Then the reduced Euler characteristic of the Milnor fibers is bounded in absolute value by $m-2$. 
If $\mu(f)= 2$, then the reduced Euler characteristic of the Milnor fibers is $0$. If the Milnor number is $3$, $5$, or $7$, then the reduced Euler characteristic of the Milnor fibers is $\pm 1$.
\end{theorem}

\begin{proof}
This follows from the result of Arnold, stated in \cite[Theorem 1.3]{vanStraten-Warmt} along with our main theorems \ref{thm:H-summand}, \ref{thm:rank-5}, and the ELT inequality in theorem \ref{thm:ELT}.
\end{proof}

The setup to this theorem was certainly known to Eisenbud and Levine in \cite{ELT} and Arnold \cite{Arnold-1978}. We do not claim any originality for this result. Similarly, the following theorem is one of the main results that Eisenbud and Levine sought to prove in their foundational paper \cite{ELT}. 

\begin{theorem}
Let $f: \bbA^n_\bbR \to \bbA^n_\bbR$ be a polynomial map with an isolated zero at $0$. If the rank of the EKL form of $f$ at $0$ is 3, 5, or 7 then the local topological degree of $f$ at $0$ is $\pm 1$. And if the rank of the EKL form of $f$ at $0$ is 2, then the local topological degree of $f$ at $0$ is $0$. 
\end{theorem}

\begin{proof}
This follows from Eisenbud and Levine's main result \cite[Theorem  1.2]{ELT}, our main results, theorems \ref{thm:H-summand}, \ref{thm:rank-5}, and the Eisenbud--Levine--Teissier inequality in rank 7, see theorem \ref{thm:ELT}. 
\end{proof}

The Eisenbud--Levine--Teissier inequality can of course be used to obtain more general bounds for the degree. Explicitly, Eisenbud and Levine state this in \cite[Theorem 2.1]{ELT} and give further geometric interpretations in \cite[\S 2]{ELT}.

\begin{acknowledgements}
We thank the Center for Advanced Study in Oslo for their hospitality, where we were able to complete much of this work and where we had many helpful discussions with Sabrina Pauli. 
We are also very grateful to Kirsten Wickelgren who suggested the applications of our results in section \ref{sec:real} and her enthusiastic support for our project. 
Our thanks also go to Stephen McKean for helpful discussions and to the anonymous referee for many helpful comments and suggestions. 
We are also grateful for the support from the Motivic Hopf Equations project RCN no.\ 250399.
\end{acknowledgements}

\bibliography{bibliography}
\bibliographystyle{mscplain}

\end{document}